\documentclass[11pt]{amsart}
\usepackage{srcltx}

\usepackage{amssymb}
\usepackage{amsmath}
\usepackage{amsthm}
\usepackage{amscd}
\usepackage{graphicx}
\usepackage{amsfonts}
\newtheorem{thm}{Theorem}[section]

\newtheorem{lem}[thm]{Lemma}

\theoremstyle{definition}
\newtheorem{defn}[thm]{Definition}
\theoremstyle{remark}
\newtheorem{rem}[thm]{Remark}

\newcommand{\id }{{\rm id}}
\newcommand{\e }{\varepsilon }

\renewcommand{\phi }{\varphi}
\newcommand{\Ker }{{\rm Ker }}
\renewcommand{\ll }{\langle\hspace{-.7mm}\langle }
\newcommand{\rr }{\rangle\hspace{-.7mm}\rangle }
\renewcommand{\d }{{\rm def} }
\newcommand{\vd }{{\rm vd}_{p}}

\newcommand{\colim}{\operatorname{colim}}
\newcommand{\caln}{{\mathcal N}}
\begin{document}

\title{Approximating the first $L^2$-betti number of residually finite groups}
\author{W. L\"uck, D. Osin}\thanks{This paper is financially supported by the Leibniz-award of the first author. The research of the second author was supported by
the NSF grant  DMS-1006345.}
        \address{Mathematicians Institut der Universit\"at Bonn\\
                Endenicher Allee 60\\
                53115 Bonn, Germany}
         \email{wolfgang.lueck@him.uni-bonn.de}
\address{Department of Mathematics, Vanderbilt University, Nashville, TN 37240, U.S.A.}
\email{denis.v.osin@vanderbilt.edu}
\date{}

\begin{abstract}
We show that the first $L^2$-betti number of a finitely generated residually finite group can be estimated from below  by using ordinary first betti numbers of finite index normal subgroups. As an application we construct a finitely generated infinite residually finite torsion group with positive first $L^2$-betti number.
\end{abstract}

\maketitle

\section{Introduction}

Let $G$ be a finitely generated residually finite group $G$ and let $\{ N_i\}_{i\in \mathbb N}$ be a sequence of finite index normal subgroups of $G$ such that
\begin{equation}\label{N}
N_1\ge N_2 \ge \ldots ,\;\;\; {\rm and}\;\;\; \bigcap\limits_{i=1}^\infty N_i=\{ 1\} .
\end{equation}
The approximation theorem proved by the first author in \cite{Lueck94} implies that if $G$ is finitely presented, then
\begin{equation}\label{approx}
b_1^{(2)} (G)= \lim \limits_{i\to \infty} \frac{b _1 (N_i)}{[G:N_i]}.
\end{equation}
where $b_1^{(2)} (G)$ is the first $L^2$-betti number of $G$. In particular the limit in (\ref{approx}) exists and is independent of the choice of the sequence of normal subgroups satisfying (\ref{N}).

The question of whether (\ref{approx}) holds for any finitely generated group was open until now. It is partially motivated by some other open problems in group theory. For instance, the affirmative answer would imply the existence of a non-residually finite hyperbolic group \cite{Osin09} and would disprove either the cost vs first $L^2$-betti number conjecture or the fixed price conjecture (see \cite{Osin10} for details).

In this paper we first prove the following.

\begin{thm}\label{main1}
Let $G$ be a finitely generated residually finite group. For every sequence of finite index normal subgroups $\{ N_i\}_{i\in \mathbb N}$ of $G$ satisfying (\ref{N}), we have
\begin{equation}\label{ineq}
b _1^{(2)} (G)\ge \lim \limits_{i\to \infty}\sup \frac{b _1 (N_i)}{[G:N_i]}.
\end{equation}
\end{thm}

Unfortunately Theorem \ref{main1} is not sufficient for the above-mentioned  applications. In fact, the opposite inequality would be sufficient, but our next result shows that it does not hold. Observe that the the right side of (\ref{approx}) equals $0$ for any sequence $\{ N_i\} $ whenever $G$ is a torsion group.

\begin{thm}\label{main2}
For every prime $p$, there exists a finitely generated infinite residually finite $p$-group with positive first $L^2$-betti number.
\end{thm}

Our interest in groups constructed in Theorem \ref{main2} is also motivated by von-Neumann-type problems. Recall that the original von Neumann problem (sometimes referred to as the von Neumann--Day problem) asks whether there exist  non-amenable groups without non-abelian free subgroups. This question was first answered affirmatively by Olshanskii in \cite{Ols} and since then many other examples have been constructed including residually finite \cite{Ers} and finitely presented ones \cite{OS}. Similar problems were considered for groups satisfying other conditions close to non-amenability (see, e.g., \cite{EM}, \cite{Osin09}). Interestingly, a result of Lackenby \cite[Theorem 1.6]{Lac} implies that if a finitely presented infinite group has positive first $L^2$-betti number (which can be thought of as an extreme form of non-amenability) and is residually $p$-finite, then it does contain non-abelian free subgroups. Our Theorem \ref{main2} shows that this result can not be extended to all finitely generated groups.

\section{Estimating the first $L^2$-betti number from below}

In this section we prove Theorem \ref{main1}. The main ingredient of the proof is the result of the first author stating that the approximation conjecture holds for groups from a class $\mathcal G$, which includes in particular all residually finite groups. We refer to \cite[Chapter 13]{Lueck} for more details.

\begin{proof}[Proof of Theorem \ref{main1}]
  Choose a presentation
  \[G = \langle s_1,s_2, \ldots s_g\mid R_1, R_2, \ldots \rangle,\]
  where the
  number of generators is finite.  For
  a natural number $j$, let $G_j$ be the finitely presented group given by the
  presentation
  \[
   G = \langle s_1,s_2, \ldots s_g\mid R_1, R_2, \ldots R_j\rangle.
   \]
  Let $\psi_i \colon G_i \to G$ and $\phi_{j,k} \colon G_j \to G_k$ for $j \le k$ be
  the obvious projections. We have
  \[
   G = \colim_{j} G_j.
   \]

   The system of group homomorphisms
  \[G_0 \xrightarrow{\phi_{0,1}} G_1 \xrightarrow{\phi_{1,2}} G_2
  \xrightarrow{\phi_{2,3}} \cdots\] induces a system of maps of $CW$-complexes
  \[BG_0 \xrightarrow{\psi_{0,1}} BG_1 \xrightarrow{\psi_{1,2}} BG_2
  \xrightarrow{\psi_{2,3}} \cdots\] We can arrange that the $2$-skeleton of
  $BG_j$ is finite for all $j \ge 0$ since each $G_j$ is finitely presented.
  Let $X$ be the infinite mapping telescope of this system. It is a
  $CW$-complex.  Since we have for $k \ge 0$
  \[
  \pi_k(X) = \colim_j \pi_k(BG_j)
  \]
  we conclude
  \[
  \pi_k(X) =
  \begin{cases} \{1\} & k \ge 2;\\ G & k = 1.
  \end{cases}
  \]
  Hence $X$ is  a model for $BG$.

  Let $X_j$ be the mapping telescope of the
  finite system
  \[
  BG_0 \xrightarrow{\psi_{0,1}} BG_1 \xrightarrow{\psi_{1,2}} BG_2
  \xrightarrow{\psi_{2,3}} \cdots \xrightarrow{\psi_{j-1,j}} BG_j
  \]
  Then the $2$-skeleton of $X_j$ is finite and we have the nested sequence of
  $CW$-subcomplexes
  \[X_0 \subseteq X_1 \subseteq X_2 \subseteq \cdots\] of $X$ satisfying
  $X = \bigcup_{j} X_j$.  Let $\widetilde{X} \to X$ be the universal covering of
  $X$. This is a model for the universal principal $G$-bundle $G \to EG \to BG$.
  Let $\widetilde{X_n}$ be its restriction to $X_n$.  We obtain a nested
  sequence of free $G$-$CW$-subcomplexes
  \[
  \widetilde{X}_0 \subseteq \widetilde{X}_1 \subseteq \widetilde{X}_2
  \subseteq \cdots
  \]
  of $\widetilde{X}$ satisfying $\widetilde{X} = \bigcup_{j}  \widetilde{X}_j$.

  Since taking homology and taking tensor products are compatible
  with directed colimits, we get
  \[
  H_1\bigl(\widetilde{X_n};\caln(G)\bigr)
  = \colim_{j} H_1\bigl(\widetilde{X};\caln(G)\bigr).
  \]
  Since the $G$-$CW$-complex $\widetilde{X}_j$ has only finitely many orbits of cells of dimension at most $2$ for all $j \ge 0$, we get $\dim_{\caln(G)}\bigl(H_1(\widetilde X_j;\caln(G)\bigr)< \infty$ for all $j \ge 0$.
  Each map $\phi_{j-1,j} \colon G_{j-1} \to G_j$ is surjective. Hence each map
  $\psi_{j-1,j} \colon BG_{j-1} \to BG_j$ is $1$-connected. Therefore
  each inclusion $\widetilde{X}_{j-1} \to \widetilde{X}_j$ is $1$-connected.
  This implies that the induced map
  $H_1\bigl(\widetilde{X}_{j-1},\caln(G)\bigr) \to H_1\bigl(\widetilde X_{j},\caln(G)\bigr)$
  is surjective. Hence we get for $j \ge 1$
  \[ \dim_{\caln(G)}\left(H_1\bigl(\widetilde{X}_{j-1};\caln(G)\bigr)\right)
  \ge  \dim_{\caln(G)}\left(H_1\bigl(\widetilde{X}_j;\caln(G)\bigr)\right).
  \]
  We conclude
  from~\cite[Theorem~6.13~(2) on page~243]{Lueck}
  \begin{eqnarray}
  \dim_{\caln(G)}\left(H_1\bigl(\widetilde{X};\caln(G)\bigr)\right)
  & = &
  \lim_{j} \dim_{\caln(G)}\left(H_1\bigl(\widetilde{X}_j;\caln(G)\bigr)\right).
  \label{lim_b_1(2)}
\end{eqnarray}

  Consider a nested sequence
  $G = N_0 \supseteq N_1 \supseteq N_2 \supseteq N_3 \supseteq \cdots$ of normal
  subgroups of finite index such that $\bigcap_{i = 0}^{\infty} N_i = \{1\}$.
  Since $G$ is residually finite it satisfies the Approximation Conjecture (see \cite[Conjecture 13.1]{Lueck}), and we obtain
  \begin{eqnarray}
  \dim_{\caln(G)}\left(H_1\bigl(\widetilde{X}_j;\caln(G)\bigr)\right)
  & = &
  \lim_{i} \frac{b_1(N_i\backslash \widetilde{X}_j)}{[G : N_i]}.
   \label{approxi_fin_pres}
  \end{eqnarray}
Indeed this follows from ~\cite[Theorem~13.3 on page~454]{Lueck} applied to the $2$-skeleton of $\widetilde{X}_j$ (which has only finitely many orbits of cells) and the observation that both sides of (\ref{approxi_fin_pres}) only depend on the $2$-skeleton.

  Since $\widetilde{X}_j \to \widetilde{X}$ is $1$-connected, the map
 $N_i\backslash \widetilde{X}_j \to N_i\backslash \widetilde{X}$ is $1$-connected
   and hence $b_1(N_i\backslash \widetilde{X}_j) \ge b_1(N_i\backslash \widetilde{X})$
   holds for all $j \ge 0$. This implies for all $j \ge 0$
   \begin{eqnarray}
  \lim_{i} \frac{b_1(N_i\backslash \widetilde{X}_j)}{[G : N_i]}
  & \ge &
  \limsup_{i} \frac{b_1(N_i\backslash \widetilde{X})}{[G : N_i]}.
    \label{lim_greater_equal_limsup}
  \end{eqnarray}
  Since the $G$-$CW$-complex $\widetilde{X}$ is a model for $EG$, we get
  $b_1(N_i\backslash \widetilde{X}) = b_1(N_i)$ for all $i$ and
  $b_1^{(2)}(G) = \dim_{\caln(G)}\left(H_1\bigl(\widetilde{X};\caln(G)\bigr)\right)$.
  We conclude from~\eqref{approxi_fin_pres}
  and~\eqref{lim_greater_equal_limsup} for all $j$
  \begin{eqnarray}
  \dim_{\caln(G)}\bigl(\widetilde{X}_j;\caln(G)\bigr)
  & \ge &
  \limsup_{i} \frac{b_1(N_i)}{[G : N_i]}
  \label{nearly_ready}
  \end{eqnarray}
  Now~\eqref{lim_b_1(2)} and~\eqref{nearly_ready} imply
  \[b_1^{(2)}\bigl(G) \ge  \limsup_{i \to \infty} \frac{b_1(N_i)}{[G:N_i]}.\]\end{proof}

\section{Virtual deficiency of finitely presented groups}

In what follows, ``$p$-finite" always means ``equal to a power of $p$". Given two elements $x,y$ of a group $G$, we  write $x^y$ for $y^{-1}xy$. We denote by $\ll S\rr ^G$ (or just $\ll S\rr $ if no confusion is
possible) the normal closure of a subset $S$ in $G$, i.e., the smallest normal subgroup of $G$ containing $S$. Finally if $G$ is finitely presented, $\d (G)$ denotes the deficiency of $G$, i.e., the maximum of the difference between the number of generators and the number of relations over all finite presentations of $G$.

Let $G=F/R$ be a finitely presented group, where $F=\langle x_1, \ldots , x_d\rangle $ is free of rank $d$ and $R=\ll R_1, \ldots , R_r\rr ^F$. Let $H$ be a finite index subgroup of $G$, $K$ the full preimage of $H$ in $F$. By the Nilsen-Schreier formula $K$ is a free group of rank $(d-1)j+1$, where $j=[F:K]=[G:H]$. It is straightforward to check that $R=\ll R_i^t \mid i=1, \ldots, r,\, t\in T\rr ^K$, where $T$ is a set of left coset representatives of $K$ in $F$. Thus $H=K/R$ has a presentation with $(d-1)j+1$ generators and $r|T|=rj$ relations. In particular,
$$
\d (H)-1\ge
(\d (G)-1)[G:H].
$$
We will refer to this property of deficiency as {\it supermultiplicativity}.

\begin{defn}
Let $G$ be a finitely presented group. We define the {\it $p$-virtual deficiency} of $G$ by
\begin{equation}\label{vd}
\vd (G)= \sup\frac{\d (H)-1}{[G:H]},
\end{equation}
where the supremum is taken over all normal subgroups $H\le G$ of $p$-finite index.
\end{defn}

\begin{rem}
Clearly every group of positive $p$-virtual deficiency is infinite. Using supermultiplicativity it is easy to show that the definition does not change if we take the supremum over all (not necessarily normal) subgroups of $p$-finite index in $G$. We do not know if the supremum in (\ref{vd}) is always achieved and whether it is always rational. Some applications of a similar notion of $p$-deficiency can be found in \cite{BT}.
\end{rem}

The following lemma was proved in \cite[Lemma 2.3]{Osin10}. (The proof is an easy exercise.)

\begin{lem}\label{L23}
Let $G$ be a finitely presented group, $N$ a finite index normal subgroup of $G$, $g$ an element of $G$, $m$ the order of $gN$ in $G/N$. Let $M$ denote the natural image of $N$ in the quotient group $G/\ll g^m\rr $. Then $$\d (M)\ge \d (N) - [G:N]/m .$$
\end{lem}

Lemma \ref{L23} can be used to construct nontrivial examples of groups with positive $p$-virtual deficiency. For a group $G$, we denote by $\widehat G_p$ the image of $G$ in its pro-$p$ completion. That is $\widehat G_p=G/R$, where $R$ is the intersection of all normal subgroups $N$ of $p$-finite index in $G$.

\begin{lem}\label{quot}
Let $G$ be a finitely presented group. Suppose that the image of an element $g\in G$ in $\widehat G_p$ has infinite order. Then for every $\delta >0$ there exist arbitrary large integers $n$ such that
\begin{equation}\label{vdG}
\vd\left(G/\ll g^{p^n}\rr \right)\ge \vd (G)-\delta .
\end{equation}
\end{lem}

\begin{proof}
Let us fix any $K>0$. We wish to find $n\ge K$ that satisfies (\ref{vdG}). Without loss of generality we can assume that
\begin{equation}\label{q1}
\frac1{p^K}\le \delta /2.
\end{equation}
Since $g$ has infinite order in $\widehat G_p$, there is a $p$-finite quotient  $Q$ of $G$ such that the order of the image of $g$ in $Q$ is at least $p^K$. Let $H$ be a normal subgroup of $p$-finite index in $G$ such that
\begin{equation}\label{q2}
\frac{\d (H)-1}{[G:H]} \ge \vd (G)-\delta/2
\end{equation}
and let $N=\Ker (G\to Q)\cap H$. Clearly $N$ is normal of $p$-finite index in $G$. Since $N\le \Ker (G\to Q)$, the order of $gN$ in $G/N$ is $p^n$ for some $n\ge K$. Let $M$ denote the image of $N$ in $G_1=G/\ll g^{p^n}\rr $. Note that $[G_1:M]=[G:N]$. Using subsequently Lemma \ref{L23}, supermultiplicativity, (\ref{q1}), and (\ref{q2}) we obtain
$$
\frac{\d (M)-1}{[G_1:M]}\ge \frac{\d (N)- 1}{[G:N]} - \frac1{p^n}\ge \frac{\d (H)-1}{[G:H]}- \frac1{p^K}\ge  \vd (G)-\delta.
$$
\end{proof}

The next two lemmas allow us to estimate the first betti and $L^2$-betti numbers of some (not necessarily finitely presented) residually finite groups.

\begin{lem}\label{b1}
For any finitely presented group $G$, we have $b_1(\widehat G_p)\ge \d (G)$.
\end{lem}
\begin{proof}
If $G$ has a presentation with $d$ generators and $r$ relations, then $G/[G,G]$ is a quotient of $\mathbb Z^d$ by a subgroup of rank at most $r$. It is well-known that $G/[G,G]$ maps onto $\mathbb Z^{d-r}$ in this case. Since free abelian groups are residually $p$-finite, $\widehat G_p$ also maps onto $\mathbb Z^{d-r}$ and hence $b _1 (\widehat G_p)\ge d-r$.
\end{proof}

\begin{lem}\label{b1Gp}
For any finitely presented group $G$, we have
\begin{equation}\label{b12G}
b_1^{(2)} (\widehat G_p )\ge \vd (G).
\end{equation}
\end{lem}

\begin{proof}
Fix some $\e>0$. Let $H$ be a subgroup of $p$-finite index in $G$ that satisfies (\ref{q2}).
Let $K$ be the image of $H$ in $\widehat G_p$ and let $M$ be any subgroup of $p$-finite index in $K$. We denote by $N$ be the full preimage of $M$ in $G$ (see the diagram below).
$$
\begin{array}{ccccc}
M&\lhd & K &\lhd & \widehat G_p\\
\uparrow && \uparrow && \uparrow\\
N&\lhd & H &\lhd & G
\end{array}
$$

Using supermultiplicativity of deficiency and (\ref{q2}), we obtain
\begin{equation}\label{dn}
\d (N)-1\ge (\d (H) -1)[H:N]\ge (\vd (G)-\e )[G:H][H:N].
\end{equation}
Since $[G:N]$ is a power of $p$, $\widehat N_p\cong M$. By (\ref{dn}) and Lemma \ref{b1} we have
\begin{equation}\label{b1m}
b _1 (M)\ge \d (N)> (\vd (G)-\e )[G:H][H:N]=(\vd (G)-\e )[\widehat G_p:K][K:M].
\end{equation}
Since $K$ is residually $p$-finite and (\ref{b1m}) holds for any subgroup $M\le K$ of $p$-finite index, we obtain
$$
b_1^{(2)} (K)\ge (\vd (G) -\e )[\widehat G_p:K]
$$
by Theorem \ref{main1}. Using multiplicativity of  $b_1^{(2)}$ and letting $\e\to 0$, we get (\ref{b12G}).
\end{proof}

\section{Residually finite $p$-groups with positive first $L^2$-betti number}

We are now ready to prove Theorem \ref{main2}. In fact, we prove a stronger result.

\begin{thm}\label{th2a}
For any prime $p$, any integer $n\ge 2$, and any $\e> 0$, there exists a finitely generated infinite residually finite $p$-group $Q$ generated by $n$ elements such that $b_1^{(2)} (Q)\ge n-1-\e$.
\end{thm}

\begin{proof}
We use a modification of the main construction from \cite{Osin10}. From now on, let us fix $p$ and denote $\widehat G_p$ simply by $\widehat G$ for a group $G$. Let $F$ be the free group of rank $n$ with basis $X$. We enumerate all elements of $F=\{ 1=f_0, f_1, \ldots \} $ and construct inductively a commutative diagram of quotient groups of $F$ and epimorphisms
\begin{equation}\label{seq}
\begin{array}{ccccc}
G_0 & \xrightarrow{\alpha _1} & G_1 & \xrightarrow{\alpha _2} & \ldots\\
\downarrow && \downarrow && \\
\widehat G_0 & \xrightarrow{\hat\alpha _1} & \widehat G_1 & \xrightarrow{\hat\alpha _2} & \ldots\\
\bigtriangledown && \bigtriangledown &&\\
N_0 && N_1 &&\ldots  ,\\
\end{array}
\end{equation}
where the vertical arrows are natural maps from groups to their images in pro-$p$ completions and for every $k\in \mathbb N\cup\{ 0\}$ the following conditions hold. For simplicity we use the same notation for elements $f_0, f_1, \ldots $ and their images in quotients of $F$.
\begin{enumerate}
\item[(a)] The order of $f_k$ in $\widehat G_k$ is $p$-finite.
\item[(b)] $G_k$ is finitely presented and $\vd (G_k) > n-1-\e $.
\item[(c)] $N_k$ is of $p$-finite index in $\widehat G_k$ and the shortest nontrivial element of $N_k$ has length at least $k$ in the word metric on $\widehat G_k$ corresponding to the natural image of $X$.
\item[(d)] If $k\ge 1$, then $\Ker (\hat \alpha _{k})\le N_{k-1}$ and $N_k\le \hat \alpha _k(N_{k-1})$.
\end{enumerate}

Properties (a)-(c) are easy to verify for $G_0=F$ and $N_0=\widehat G_0$. For $k>0$ we consider two cases.

{\it Case 1.} Suppose that the order of $f_{k}$ in
$\widehat G_{k-1}$ is finite. Note that it is $p$-finite in this case as $\widehat G_{k-1}$ is residually $p$-finite. Let $G_{k}=G_{k-1}$, $\alpha _{k}=\id $, and $\hat\alpha _{k}=\id $. Then (a) and (b) obviously hold. Since $\widehat G_{k}$ is residually $p$-finite, there exists a subgroup $N_{k}\lhd \widehat G_{k}$  of $p$-finite index that satisfies (c). We can always find such $N_{k}$ inside $\hat \alpha_{k} (N_{k-1})$ so that (d) is satisfied as well.

{\it Case 2.} Suppose that the order of  $f_{k}$ in $\widehat G_{k-1}$ is infinite. Let $p^m$ be the order of $f_k$ in $\widehat G_{k-1}/N_{k-1}$. By Lemma \ref{quot} and the inductive assumption, we can choose $n\ge m$ such that the quotient group $G_{k}=G_{k-1}/\ll f_k^{p^n}\rr $ satisfies (b). Let $\alpha _{k}$ be the natural epimorphism $G_{k-1}\to G_k$ and let $\hat \alpha _k$ be the epimorphism induced by $\alpha _k$. As $n\ge m$, $f_k^{p^n}\in N_{k-1}$ in the group $\widehat G_{k-1}$ and hence $\Ker(\hat \alpha _{k})\le N_{k-1}$. Again since $\widehat G_{k}$ is residually finite, there exists a normal subgroup $N_k\le \alpha_k (N_{k-1})$ of $p$-finite index in $\widehat G_{k}$ that satisfies (c). This completes the inductive step.

Let now $Q$ be the (co)limit of the second row in (\ref{seq}). Condition (a) obviously implies that $Q$ is a torsion $p$-group. By (c) and (d) $Q$ is residually finite. Indeed if $q\in Q$ is a nontrivial element of length $k$ with respect to the word metric corresponding to the natural image of $X$ and $g$ is a shortest preimage of $q$ in $\widehat G_{k+1}$, then $g\notin N_{k+1}$ by (c). Note that (d) implies $\Ker (\widehat G_{k+1}\to Q)\le N_{k+1}$. Hence $Q$ maps onto $\widehat G_{k+1}/N_{k+1}$ and $q$ is taken to $g\ne 1$ by this map. Observe also that by (b) all groups $\widehat G_k$ are infinite and hence $Q$ is infinite as well.

Finally we note that Lemma \ref{b1Gp} and property (b) imply that $b_1^{(2)} (\widehat G_k)\ge n-1-\e$. By semi-continuity of the first $L^2$-Betti number (see \cite[Theorem 1]{Pich})  we obtain
$$
b_1^{(2)} (Q)\ge \lim\limits _{k\to \infty} \sup b_1^{(2)} (\widehat G_k)\ge n-1-\e.
$$
\end{proof}

\end{document}